\documentclass[a4paper,12pt]{amsart}

\usepackage{todonotes}

\usepackage{amsfonts}
\usepackage{graphicx}
\usepackage{float}
\usepackage{wrapfig}
\usepackage{subfigure}

\usepackage{url}

\usepackage{hyperref}
\usepackage{geometry}

\usepackage{verbatim}

\usepackage[english]{babel}
\usepackage[utf8x]{inputenc} 
\usepackage[T1]{fontenc}

\usepackage{amssymb,amsmath}
\usepackage{amsthm}
\usepackage{bbm}

\usepackage{mathrsfs}

\usepackage{color}
\usepackage[color,all]{xy}

\theoremstyle{definition}
\newtheorem{definition}{Definition}[section]

\theoremstyle{plain}
\newtheorem{theorem}[definition]{Theorem}

\newtheorem{case}{Case}

\newtheorem{conjecture}[definition]{Conjecture}
\newtheorem{corollary}[definition]{Corollary}

\newtheorem{lemma}[definition]{Lemma}

\theoremstyle{remark}
\newtheorem{example}[definition]{Example}
\newtheorem{remark}[definition]{Remark}
\newtheorem*{claim}{Claim}

\newcommand{\OO}[1]{\mathcal{O}_{#1}}

\newcommand{\PP}{\mathbb{P}}

\newcommand{\Lscr}{L} 
\newcommand{\Lcal}{\mathcal{L}}
\newcommand{\Mscr}{\mathcal{M}}
\newcommand{\Fscr}{\mathscr{F}}
\newcommand{\OC}{\text{\upshape{O\!C}}}
\newcommand{\mfrak}{\mathfrak{m}}

\newcommand{\Wrd}{W^r_d}

\newcommand{\Spec}{\text{\upshape{Spec}}}

\newcommand{\coker}{\text{\upshape{coker}}}
\newcommand{\Supp}{\text{\upshape{Supp}}}
\newcommand{\Sing}{\text{\upshape{Sing}}}

\newcommand{\Hom}{\text{\upshape{Hom}}}

\newcommand{\rank}{\text{\upshape{rank}}}

\newcommand{\Rat}{\text{\upshape{Rat}}}
\newcommand{\Prin}{\text{\upshape{Prin}}}

\newcommand{\Pic}{\text{\upshape{Pic}}}

\begin{document}

\title{The Osculating cone to special Brill-Noether loci}
\author{Michael Hoff}
\address{Universit\"at des Saarlandes, Campus E2 4, D-66123 Saarbr\"ucken, Germany}
\email{hahn@math.uni-sb.de}
\author{Ulrike Mayer}
\email{mayer@math.uni-sb.de}


\begin{abstract}
We describe the osculating cone to Brill-Noether loci $W^0_d(C)$ 
at smooth isolated points of  $W^1_d(C)$ for a general canonically embedded curve $C$ of even genus $g=2(d-1)$. 
In particular, we show that the canonical curve $C$ is a component of the osculating cone.
The proof is based on techniques introduced by George Kempf.
\end{abstract}

\maketitle

\section{Introduction}
\label{introduction}

In Brill-Noether theory, one studies the geometry of Brill-Noether loci $W^r_d(C)$ for a curve $C$, i.e., 
schemes whose closed points consist of the set 
$$
\{ \Lscr \in \Pic(C)\mid \deg \Lscr =d \text{ and } h^0(C,\Lscr)\geq r+1 \}\subset \Pic^d(C)
$$
of linear series of degree $d$ and dimension at least $r+1$. 
For a general curve, the dimension of $\Wrd(C)$ is equal to the Brill-Noether Number 
$
\rho(g,d,r)=g-(r+1)(g-d+r)
$ 
by the Brill-Noether Theorem (see \cite{GH80}). 
In this paper, we study the local structure of $W^0_d(C)$ near points of $W^1_d(C)$.

The Brill-Noether locus $W^0_d(C)\subset \Pic^d(C)$ is given locally around a
point $\Lscr\in \Wrd(C)$ with $h^0(C,\Lscr)=r+1$ by the maximal minors of the matrix
$(f_{ij})$ of regular functions vanishing at $\Lscr$ arising from
\vspace{-2mm}
$$
R\pi_* \mathscr{L}: 0 \longrightarrow \OO {\Pic^d(C),\Lscr}^{h^0(C,\Lscr)} \stackrel{(f_{ij})}{\longrightarrow} 
\OO {\Pic^d(C),\Lscr}^{h^1(C,\Lscr)} \longrightarrow 0,
$$
the only non-trivial part of the direct image complex of the Poincar\'{e} bundle $\mathscr{L}$ 
on $C\times \Spec(\OO {\Pic^d(C),\Lscr})$ (see \cite{Kem83}).
Using flat coordinates on the universal covering of $\Pic^d(C)\cong H^1(C,\OO C)/ H^1(C,\mathbb{Z})$, 
we may expand 
$$
f_{ij}\ =\ l_{ij}\  +\  q_{ij}\  +\  \emph{higher order terms},
$$
where $l_{ij}$ and $q_{ij}$ are linear and quadratic forms on the tangent space $H^1(C,\OO C)$ of 
$\Pic(C)$ at $\Lscr$, respectively.
A first local approximation of $W^0_d(C)$ is the tangent cone 
$\mathcal{T}_{\Lscr}(W^0_d(C))$ whose ideal is generated by the maximal minors of $(l_{ij})$ 
by Riemann-Kempf's Singularity Theorem (see \cite{Kem73}). 
Recall that the analytic type of $W^0_d(C)$ at $\Lscr$ is completely determined by the tangent cone. 
As its subvariety, we will study the osculating cone of order $3$ to $W^0_d(C)$ at the point $\Lscr$, 
denoted by $\OC_3 (W^0_d(C),\Lscr)$, a better approximation than the tangent cone 
in the given embedding into $H^1(C,\OO C)$.

In \cite{Kem86}, Kempf showed for a canonically embedded curve of genus $4$, 
the osculating cone of order $3$ to $W^0_3(C)$ at a point $\Lscr\in W^1_3(C)$ coincides with the curve. 
Using Kempf's cohomology obstruction theory, 
Kempf and Schreyer (\cite{KS88}) studied the osculating cone to $W^0_{g-1}(C)$ for curve of genus $g\geq 5$  
and proved a local Torelli Theorem (see Conjecture \ref{conjecture}). 
Schreyer conjectured that the geometry of the osculating cone to other Brill-Noether loci $(d\neq g-1)$
is rich enough to recover the curve $C$. 
Using methods developed in \cite{Kem86} and \cite{KS88}, 
we will give a positive answer for smooth pencils $\Lscr\in W^1_d(C)$ with 
$h^0(C,\Lscr)=2$ and $\dim(W^1_d(C))=\rho(g,d,1)=0$.

To explain our main theorem, we introduce some notation. 
Let $C$ be a smooth canonically embedded curve of genus $g=2(d-1)\geq 4$ over 
an algebraically closed field $\Bbbk$ of characteristic different from $2$. 
For a general curve, the Brill-Noether locus $W^1_d(C)$ is then zero-dimensional, 
and for every $\Lscr\in W^1_d(C)$, the multiplication map 
\vspace{-4mm}
\begin{align}
\label{multMap}
 \mu_{\Lscr}: H^0(C,\Lscr) \otimes H^0(C,\omega_C \otimes \Lscr^{-1}) 
              \stackrel{\cong}{\longrightarrow} H^0(C, \omega_C)
\end{align}
is an isomorphism by \cite{Gie82}. 
Hence, $W^1_d(C)$ consists of $\frac{(2d-2)!}{d!\cdot (d-1)!}$ reduced isolated points.
Let $\Lscr$ be a point of $W^1_d(C)$ such that $\mu_{\Lscr}$ is an isomorphism. 

By our choice of $\Lscr$, the projectivization of the tangent cone has a simple description 
which is important for our considerations.
Recall that the projectivization of the tangent cone and the osculating cone live naturally 
in the canonical space $\PP^{g-1}:=\PP(H^0(C,\omega_C)^*)$.
By Riemann-Kempf's Singularity Theorem, $\PP \mathcal{T}_{\Lscr}(W^0_d(C))$ is geometrically 
the scroll swept out by $g^1_d=|\Lscr|$. 
It is the union of planes $\overline{D}=\PP^{d-2}$ spanned by the points of the divisor $D\in |\Lscr|$.
Furthermore, the isomorphism of the multiplication map yields that the scroll is smooth and 
hence, coincides with the Segre embedding $\PP^1\times \PP^{d-2} \subset \PP^{g-1}$. 
We get the following diagram 
\vspace{-1mm}
$$
\begin{xy}
 \xymatrix{
 \PP(\OC_3 (W^0_d(C),\Lscr)) \ar@{^(->}[r] \ar[rd] & \PP^1\times \PP^{d-2} \ar[d]^{\pi} \ar@{^(->}[r] & \PP^{g-1} \\
   & \PP^1 & 
 }
\end{xy}
\vspace{-1mm}
$$
where $\pi$ is the projection to the first component. 
A point in $\PP^1$ corresponds to a divisor $D\in |\Lscr|$ and the fiber over $D$ is 
$\pi^{-1}(D)= \overline{D}$. 
Our main theorem is a characterization of the intersection of 
the osculating cone and a fiber $\overline{D}$.

\begin{theorem}
\label{mainTheorem}
 Let $C$ be a smooth canonically embedded curve of even genus $g=2(d-1)\geq 4$ and 
 let $\Lscr \in W^1_d(C)$ such that the multiplication map $\mu_{\Lscr}$ is an isomorphism. 
 If $\text{\upshape char}(\mathbbm{k})=0$, then the fiber $\overline{D}$ of the projection $\pi$ intersects 
 the osculating cone $\PP(\OC_3(W^0_d(C),\Lscr))$ in the union of all intersections  
 $\overline{D_1}\cap \overline{D_2}$ for each decomposition $D = D_1 + D_2$ into nonzero effective divisors. 
 If $\text{\upshape char} (\mathbbm{k})> 0$, then the above is true if 
 $\pi|_C: C\rightarrow \PP^1$ is tamely ramified. 
 In particular, the osculating cone consists of $2^{d-1}-1$ points in a general fiber.
\end{theorem}

In \cite{Mayer}, the second author showed that all intersection points are contained 
in the osculating cone for a general fiber.
For $g=4$, this generalizes the main theorem of \cite{Kem86}. 

An immediate consequence of Theorem \ref{mainTheorem} is a Torelli-type Theorem for 
osculating cones to $W^0_d(C)$ at isolated singularities.

\begin{corollary}
 The general canonical curve $C$ of genus $g=2(d-1)\geq 4$ is an irreducible component of 
 the osculating cone $\PP(\OC_3(W^0_d(C),\Lscr))$ of order $3$ to $W^0_d(C)$ at an arbitrary point $\Lscr\in W^1_d(C)$.
\end{corollary}

In the case of $W^1_d(C)$ of positive dimension, we have  
the following local Torelli-type conjecture for Brill-Noether loci.  

\begin{conjecture}
\label{conjecture}
Let $C$ be a general canonically embedded curve of genus $g$ and let $\Lscr\in W^1_d(C)$ be a general point
where $\dim (W^1_d(C))\geq 1$. Let $V=\Sing(\PP \mathcal{T}_{\Lscr}(W^0_d(C)))$. The projection 
$
\pi_V:C\to C'\subseteq \PP^1\times \PP^{h^1(C,\Lscr)-1}
$ 
is birational to the image of $C$. Consider 
 $$
\begin{xy}
 \xymatrix{
 \widetilde{\PP (\OC_3 (W^0_d(C),\Lscr))} \ar[dr]^{\alpha} \ar@{^(->}[r] 
  & \widetilde{\PP \mathcal{T}_{\Lscr}(W^0_d(C))} \ar@{^(->}[r] \ar[d] 
  & \widetilde{\PP^{g-1}} \ar[d]^{\pi_V} \ar[r] 
  & \PP^{g-1}\\
  & \PP^1\times \PP^{ h^1(C,\Lscr)-1}  \ar[r] 
  & \PP^{ 2\cdot h^1(C,\Lscr)-1} &
 }
\end{xy}
$$ 
where $\sim$ denotes the strict transform after blowing up $V$ and 
the vertical maps are induced by the projection from $V$. 
Then, 
\begin{itemize}
 \item [(a)] away from points of $C'$ the fibers of $\alpha$ are smooth or empty and 
 \item [(b)] for a smooth point $p'$ of $C'$ the corresponding point $p$ of $\widetilde{C}$ is 
             the only singular point of the fiber of $\alpha$ over $p'$.
\end{itemize}
\end{conjecture}

The origin of the conjecture is the work \cite{KS88} of Kempf and Schreyer. 
The authors proved the conjecture in the case $d=g-1$, 
where $W^0_{g-1}(C)$ is isomorphic to the theta divisor, thus implying a Torelli Theorem for general curves. 
Using the theory of foci, similar results are shown in \cite{CS95}, \cite{CS00}. 

Our local Torelli-type Theorem recovers the original curve as a component of the osculating cone. 
Moreover, we can identify different components of the osculating cone.  
Let $C$ and $\Lscr$ be as in Theorem \ref{mainTheorem}.
Let further $\overline{D}$ be a general fiber of $\pi$ spanned by the divisor $D\in |\Lscr|$ on $C$ of degree $d$. 
Then, the osculating cone consists of $2^{d-1}-1$ points in the fiber $\overline{D}$, 
since there are $2^{d-1}-1$ decompositions of $D$ into nonzero effective divisors 
and each intersection is zero-dimensional.
For $i\leq \lfloor \frac{d}{2} \rfloor$, 
there are $\binom{d}{i}$ points of the osculating cone in the fiber $\overline{D}$ if $i< \frac{d}{2}$ and 
$\frac{1}{2}\binom{d}{i}$ points if $i=\frac{d}{2}$, arising from 
decompositions $D=D_1+D_2$ where $\deg(D_1)=i$ or $\deg(D_1)=d-i$. 
For every $i$, we get a curve $C_i$ as the closure of the union of these points. 
The decomposition of the curve $C_i$ depends on the monodromy group of the covering $\pi|_C: C \to \PP^1$.  
A curve $C_i$ is irreducible if the monodromy group acts transitively on $C_i$. 
Recall that the monodromy group is the full symmetric group $S_d$  for a general curve, thus 
the osculating cone contains a union of $\lfloor \frac{d}{2} \rfloor$ irreducible curves. 
An easy example (see \cite[Proposition 4.1 (b)]{Dal85}) 
where an additional component of the osculating cone decomposes is the following.
For a tetragonal curve of genus $6$ with monodromy group $\mathbb{Z}_4$, 
the trigonal curve $C_2$, with the above notation, decomposes in a rational and a hyperelliptic curve. 

It can also happen that there are further even higher-dimensional components of the osculating cone 
contained in special fibers of $\pi$.

\medskip

The paper is organized as follows. Section \ref{preliminaries} provides basic lemmata which we need later on and 
motivates our main theorem.
In Section \ref{obstructionTheory}, we recall Kempf's cohomology obstruction theory. 
Using this theory, we will give a proof of Theorem \ref{mainTheorem} in Section \ref{proofOfTheMainTheorem}.

 \subsection*{Acknowledgements}
We thank Frank-Olaf Schreyer for bringing this topic to our attention as well as 
for his encouragement and useful conversations.  
The first author was supported by the DFG-grant SPP 1489 Schr. 307/5-2.


\section{Preliminaries and Motivation}
\label{preliminaries}

Throughout this paper, we fix the following notation: 
Let $C$ be a smooth canonically embedded curve of even genus 
$g=2(d-1)\geq 4$ over an algebraically closed field of characteristic $\neq 2$ and 
let $\Lscr\in W^1_d(C)$ be an isolated smooth point of the Brill-Noether locus. 

We refer to \cite{ACGH} for basic results of Brill-Noether theory. 
We deduce two simple lemmata from our assumptions on the pair $(C,L)$.

\begin{lemma}
\label{propertyCupProduct}
The linear system $|\Lscr|$ is a base point free pencil with 
$H^1(C,\Lscr^2)=0$ and surjective multiplication map $\mu_{\Lscr}$.
\end{lemma}

\begin{proof}
 Since $\Lscr$ is a smooth point and $g=2(d-1)$, the multiplication map 
 $
 \mu_{\Lscr}: H^0(C,\Lscr) \otimes H^0(C,\omega_C \otimes \Lscr^{-1}) \rightarrow H^0(C, \omega_C)
 $
 is an isomorphism between vector spaces of the same dimension.  
 Furthermore, $|\Lscr|$ is a base point free pencil since $H^0(C,\omega_C)$ has no base points. 
 By the base point free pencil trick, we deduce the vanishing 
 $H^1(C,\Lscr^2)^{*}=H^0(C,\omega_C\otimes \Lscr^{-2})=\ker(\mu)=0$ from the exact sequence 
 $$
 0 \longrightarrow \bigwedge^2 H^0(C,\Lscr)\otimes \omega_C \otimes \Lscr^{-2} \longrightarrow 
 H^0(C,\Lscr) \otimes \omega_C \otimes \Lscr^{-1} \longrightarrow \omega_C \longrightarrow 0.
 $$
\end{proof}

\begin{lemma}
\label{propertyL2}
 For any point $p\in C$, we have 
 $$
 H^0(C,\Lscr^2(-p)) = H^0(C,\Lscr).
 $$
\end{lemma}

\begin{proof}
 Let $f_0\in H^0(C,\Lscr)$ be a section vanishing at $p$. 
 The section $f_0$ is unique since $|\Lscr|$ is base point free. 
 Let $D$ be the divisor of zeros of $f_0$. 
 We get an isomorphism $\Lscr \cong \OO C(D)$, where the section $f_0$ corresponds to $1 \in H^0(C,\OO C(D))$.
 
 We compute the vector space $H^0(C, \OO C (2D - p))$. 
 Since $h^0(C,\OO C(D)) = 2$ and $|D|$ is base point free, 
 we find a rational function $h \in H^0(C,\OO C(D))$ 
 whose divisor of poles is exactly $D$.
 The Riemann-Roch theorem states  
 $$
 h^0(C, \OO C (2D)) - h^1(C, \OO C (2D)) = 2\left(\frac{g}{2} + 1\right) + 1 - g = 3,
 $$
 and since $h^1(C, \OO C (2D)) = 0$ by Lemma \ref{propertyCupProduct}, 
 a basis of $H^0(C, \OO C (2D))$ is given by $(1, h, h^2)$. 
 We conclude that
 $$
 H^0(C, \OO C (2D - p)) = H^0(C, \OO C (D))
 $$
 since $h^2\notin H^0(C, \OO C (2D - p))$.
\end{proof}

As already pointed out, the tangent cone coincides with the scroll swept out by $g^1_d=|\Lscr|$. 
In this setting, a further important object is the multiplication or cup-product map.  

\begin{remark}
\label{tangentCone}
 Let $(l_{ij})$ be the linear part of the $(d-1)\times 2$ matrix $(f_{ij})$ as in the introduction 
 (see also Corollary \ref{localEquationsW0d}).
 The matrix $(l_{ij})$ is closely related to the cup-product map $\mu_{\Lscr}$. 
 Indeed, by \cite[Lemma 10.3 and 10.6]{Kem83}, the matrix $(l_{ij})$ 
 describes the induced cup-product action 
 $$
 \cup: H^1(C, \OO C) \rightarrow \Hom(H^0(C, \Lscr), H^1(C, \Lscr)),\ b \mapsto \cup b.
 $$
 
 Thus, the tangent cone corresponds to cohomology classes $b\in H^1(C, \OO C)$ 
 such that the map $\cup b$ has rank $\leq 1$. 
 By Lemma \ref{propertyCupProduct}, the multiplication map is surjective and 
 the homomorphism $\cup b$ always has rank $\geq 1$ for $b\neq 0$. 
 We see that the map $\cup b$ has rank equal to $1$ for points in the tangent cone.
\end{remark}

Now, we will describe the osculating cone as a degeneracy locus of a map of vector bundles on 
$\PP^1\times \PP^{d-2}$. This yields the expected number of points in the intersection of 
the osculating cone and a fiber over $\PP^1$ 
and motivates our main theorem.

Let $\mathbbm{k}[t_0,t_1]$ and $\mathbbm{k}[u_0,\dots, u_{d-2}]$ be the coordinate rings of $\PP^1$ and $\PP^{d-2}$, respectively. 
Using these coordinates and the Segre embedding $\PP^1\times \PP^{d-2}\subset \PP^{g-1}$, 
the linear matrix $(l_{ij})$ can be expressed as the matrix $(t_i\cdot u_j)$ in Cox coordinates 
$\mathbbm{k}[t_0,t_1]\otimes \mathbbm{k}[u_0,\ldots,u_{d-2}]$. 
Let $(q_{ij})$ be the quadratic part of the expansion of $(f_{ij})$ in homogeneous 
forms on $\PP^{g-1}$.
Then, the ideal of the osculating cone is generated by homogeneous elements of bidegree $(3,3)$ in 
the ideal  
\begin{align*}& I_{2\times 2}\left(\begin{pmatrix} l_{ij} + q_{ij} \end{pmatrix}\right)\\
            =&I_{2\times 2}\left(\begin{pmatrix}
                            t_0u_0+q_{00} & \ldots & t_0u_{d-2}+q_{0(d-2)}\\
                            t_1u_0+q_{00} & \ldots & t_1u_{d-2}+q_{1(d-2)}
                            \end{pmatrix}^T\right)\\
            =&I_{2\times 2}\Bigg(\underbrace{\begin{pmatrix}
                            u_0 & \ldots & u_{d-2}\\
                            t_0q_{10}-t_1q_{00} & \ldots &t_0q_{1(d-2)}-t_1q_{0(d-2)}
                            \end{pmatrix}^T}_{=:A_{(t_0,t_1)}}\Bigg)
\end{align*}
where $q_{ij}$ is of bidegree $(2,2)$ in $\mathbbm{k}[t_0,t_1]\otimes \mathbbm{k}[u_0,\ldots,u_{d-2}]$ 
(see Definition \ref{oscCone}). 
In other words, the matrix $A_{(t_0,t_1)}$ induces a map between vector bundles such that the osculating cone 
is the degeneracy locus of $A_{(t_0,t_1)}$. 

We assume that the degeneracy locus has expected codimension $d-2$ in $\PP^1\times \PP^{d-2}$, i.e., 
the osculating cone is a curve.  
We may determine the expected number of points in a general fiber over $\PP^1$ 
with the help of a Chern class computation. 
Indeed, for a fixed point $(\lambda,\mu)\in \PP^1$, the osculating cone is given by the finite set 
$$
\{ p\in \PP^{d-2}|\ \rank\Bigg(A_{(\lambda,\mu)}(p)=\begin{pmatrix}
                            u_0(p) & \ldots & u_{d-2}(p)\\
                            q_0(p) & \ldots & q_{d-2}(p)
                            \end{pmatrix}^T\Bigg) < 2
\}
$$
of points in $\PP^{d-2}$ where $q_i\in \mathbbm{k}[u_0,\dots, u_{d-2}]$ 
is the polynomial $t_0 q_{1i}- t_1 q_{0i}$ evaluated in $(t_0,t_1)=(\lambda,\mu)$ for $i=0,\dots,d-2$. 
We compute its degree: 

We define the vector bundle $\Fscr$ as the cokernel of the first column of $A_{(\lambda,\mu)}$. 
Note that $u_0,\dots, u_{d-2}$ do not have a common zero in $\PP^{d-2}$.
Then, the vanishing locus of the section 
$
s(2): \OO {\PP^{d-2}} \longrightarrow \Fscr(2)
$
induced by the second column of $A_{(\lambda,\mu)}$ is the osculating cone in the fiber 
and its degree is the Chern class $c_{d-2}(\Fscr(2))$.  
The total Chern class of $\Fscr$ is given by 
$$
c(\Fscr)= \frac{c(\OO {\PP^{d-2}}^{d-1})}{c(\OO {\PP^{d-2}}(-1))} = \frac{1}{1-t}.
$$
We can transform the total Chern class to $c(\Fscr)=1+t+\cdots+t^{d-2}\mod t^{d-1}$. 
By the splitting principle, we find constants $\lambda_i$ such that 
$$
c(\Fscr)= \prod\limits_{i=1}^{d-2} (1 + \lambda_i t)=1+t+\cdots+t^{d-2}.
$$ 
The Chern class $c_{d-2}(\Fscr(2))$ is given by the coefficient of the $t^{d-2}$-term 
of $c(\Fscr(2))=\prod_{i=1}^{d-2} (1 + (2+ \lambda_i) t)$, i.e., $2^{d-1}-1$. 
If the osculating cone intersects $\overline{D}$ in a zero dimensional scheme, 
then the degree of $\overline{D}\cap \OC_3(W^0_d(C),\Lscr)$ is $2^{d-1}-1$, 
which is consistent with the assertion of the main theorem. 

\medskip 

Next, we compute degree and genus of the osculating cone.
As mentioned above, the osculating cone is the degeneracy locus of the matrix $A_{(t_0,t_1)}$ 
and we assume that the degeneracy locus has expected codimension $d-2$. 
Let $\widetilde{\Fscr}$ be the pullback of $\Fscr$ by the second projection $\PP^1\times \PP^{d-2}\to \PP^{d-2}$. 
Then, the vanishing locus of 
$
s(3,2): \OO {\PP^1\times \PP^{d-2}} \longrightarrow \widetilde{\Fscr}(3,2)
$ 
induced by the second column of $A_{(t_0,t_1)}$ is the osculating cone. 
We compute as above the Chern class $c_{d-2}(\widetilde{\Fscr}(3,2))$ and 
get the degree of the osculating cone as the total degree on $\PP^1\times \PP^{d-2}$,
\begin{align*}
\deg(\OC_3(W^0_d(C),\Lscr)) = & 2^{d-1}-1 + 3(\sum\limits_{i=1}^{d-2} i\cdot 2^{i-1})  \\ 
 = & (3(d-1)-4)2^{d-2} + 2.
\end{align*}
If the osculating cone has expected codimension $d-2$, the Eagon-Northcott complex resolves the osculating cone 
$\OC_3(W^0_d(C),\Lscr)\subset \PP^1\times \PP^{d-2}$. 
We can express the Hilbert polynomial $H_{\OC_3(W^0_d(C),\Lscr)}(x,y)$ of the osculating cone in terms of 
Betti numbers and twists appearing in the resolution. 
Using computer algebra software, we can compute the genus of the osculating cone. 
The genus is given as
\begin{align*}
g(\OC_3(W^0_d(C),\Lscr))= & 1 - H_{\OC_3(W^0_d(C),\Lscr)}(0,0) \\
= &1-(1+\sum_{i=1}^{d-2}(-1)^i \sum_{j=1}^i \binom{d-1}{i+1}(1-3j)\binom{d-3-i-j}{d-2})\\
= & (3(d-1)(d-2)-4)2^{d-3}+2. 
\end{align*}

\begin{example}
 \label{exampleOC}
 Let $C$ be a canonically embedded general curve of genus $6$ and 
 let $\Lscr\in W^1_4(C)$ be a smooth point.   
 We denote by $S=\mathbbm{k}[t_0,t_1]\otimes \mathbbm{k}[u_0,u_1,u_2]$ the coordinate ring of the Segre product $\PP^1\times \PP^2$. 
 As explained above, we express the ideal $I(\OC_3(W^0_4(C),\Lscr))$ of the osculating cone in $\PP^1\times \PP^2$
 as $2\times 2$ minors of the matrix 
 $$A_{(t_0,t_1)}=
 \begin{pmatrix}
  u_0 & u_1 & u_2 \\
  p_0 & p_1 & p_2
 \end{pmatrix}^T,
 $$
 where $p_i\in S_{(3,2)}$ are polynomials of bidegree $(3,2)$ for $i=0,1,2$. 
 
 For a general curve, the map $\varphi_{|\Lscr|}: C \to \PP^1$ has only simple ramification points and 
 by Theorem \ref{mainTheorem}, 
 the osculating cone has expected codimension $2$ in $\PP^1\times \PP^2$.
 By Hilbert-Burch Theorem, the minimal free resolution of $I(\OC_3(W^0_4(C),\Lscr))$ is 
 $$
 0 \longrightarrow S(-3,-4)\oplus S(-6,-5) \xrightarrow{A_{(t_0,t_1)}} S^3(-3,-3) \longrightarrow 
 I(\OC_3(W^0_4(C),\Lscr)) \longrightarrow 0.
 $$
 Thus, the osculating cone is a reducible curve of degree $22$ and genus $30$. 
 For more detailed analysis of the osculating cone see the end of Section \ref{proofOfTheMainTheorem}. 
\end{example}


\section{Kempf's cohomology obstruction theory}
\label{obstructionTheory}

First, we recall variation of cohomology to provide local equations of $W^0_d(C)$ 
at the point $\Lscr$. 
Then, we introduce flat coordinates on an arbitrary algebraic group as in \cite{Kem86} 
in order to give a precise definition of the osculating cone. 
Finally, we study infinitesimal deformations of the line bundle $\Lscr$ 
which are related to flat curves. 
This results in an explicit criterion for a point $b\in H^1(C,\OO C)$ to lie 
in the osculating cone.

\subsection{Variation of cohomology} 
\label{variationCohomology}

Let $S = \Spec(A)$ be an affine neighbourhood of $\Lscr \in \Pic^d(C)$, and 
let $\Lcal$ be the restriction of the Poincar\'e line bundle over $C\times \Pic^d(C)$ to $C\times S$. 
The idea of the variation of cohomology is to find an approximating homomorphism which computes simultaneously 
the upper-semicontinuous functions on $S$
$$
s \mapsto h^i(C\times \{s\}, \Lcal\otimes_A k(s)),\ i= 0,1,
$$
where $k(s)= A/\mfrak_s$ and $\mfrak_s$ is the maximal ideal of $s\in S$. 

\begin{theorem}[\cite{EGA} Theorem 6.10.5 or \cite{Kem83} Theorem 7.3] 
\label{approxHom}
 Let $\mathcal{M}$ be a family of invertible sheaves on $C$ parametrized by an affine scheme $S=\Spec(A)$. 
 There exist two flat $A$-modules $F$ and $G$ of finite type and an $A$-homomorphism $\alpha: F \rightarrow G$ 
 such that for all $A$-modules $M$, there are isomorphisms
 $$
 H^0(C\times S, \mathcal{M} \otimes_A M) \cong \ker(\alpha \otimes_A id_M),\ \ \ 
 H^1(C\times S, \mathcal{M} \otimes_A M) \cong \coker(\alpha \otimes_A id_M).
 $$
\end{theorem}

If we shrink $S$ to a smaller neighbourhood, we may assume that 
$F$ and $G$ are free $A$-modules of finite type by Nakayama's Lemma. 
Furthermore, we may assume that the approximating homomorphism is minimal, i.e., 
$\alpha \otimes_A k(\Lscr)$ is the zero homomorphism by \cite[Lemma 10.2]{Kem83}.
Applying Theorem \ref{approxHom} to $\mathcal{M}=\Lscr$ leads to the following corollary. 

\begin{corollary}
\label{localEquationsW0d}
 The local equations of $W^0_d(C)|_S$ at the point $\Lscr$ are given by the maximal minors of a 
 $(d-1)\times 2$ matrix $(f_{ij})$ of regular functions on $S$ which vanish at $\Lscr$. 
\end{corollary}

\begin{proof}
 Since $h^0(C,\Lscr)=2$ and $\chi(\Lcal \otimes k(s)) = \chi(\Lscr) = 2 - (d-1)$, $\forall s \in S$, 
 the $A$-modules $F$ and $G$ are free of rank $2$ and $d-1$, respectively.
\end{proof}


\subsection{Flat coordinates and osculating cones}
\label{flatCoordOscCone}

We recall the definition of the flat structure on an algebraic group and 
of the osculating cone according to \cite[Section 1 $\&$ 2]{Kem86}. 
 
In the analytic setting, a Lie group $G$ and its tangent space $\mathfrak{g}$ are related by the exponential mapping, 
which is an analytic diffeomorphism between an open neighbourhood of the identity of $G$ 
and an open neighbourhood of the origin of $\mathfrak{g}$ (see for instance \cite[Chapter II, \S 1]{H01}). 
Roughly speaking, the flat structure on an algebraic group is a truncated exponential mapping. 
We explain this in more details.
 
For an algebraic group $X$ of finite type over an algebraically closed field, let
$X_{n, x}$ be the $n$-th \emph{infinitesimal neighbourhood} of $x\in X$ which 
is given by the $(n+1)$-st power of the ideal of $x$.   
It is sufficient to define the flat structure on $X$ at the identity $e$ since we can translate this  
to any point of $X$. We denote $T$ the tangent space to $X$ at $e$. 
Recall that the identity is the fixed point of the $m$-power operation which sends $g$ to $g^m$ 
for all $m\in \mathbb{Z}$.

\begin{definition}
  The \emph{flat structure } of $n$-th order on $X$ at $e$ is given by an equivariant 
  isomorphism $\iota: T_{n, 0} \rightarrow X_{n, e}$ so that 
  the multiplication by $m$ on $T_{n, 0}$ coincides with 
  the action of the $m$-power operation on $X_{n, e}$ for all integers $m$. 
  We call a smooth subvariety $Y\subset X$ passing through $e$ \emph{flat to the $n$-th order} if 
  $\iota^{-1}(Y_{n,e})$ has the form $S_{n, 0}$ where $S$ is a linear subspace of $T$. 
\end{definition} 

\begin{remark}
\label{flatCurve}
 \begin{enumerate}
  \item [(a)] If $\text{char} (\mathbbm{k})=0$, there is a unique flat structure of order $n$ for all $n\in \mathbb{N}$. 
        In the limit, these flat structures give an analytic isomorphism between a neighbourhood of the $0$ in $T$ 
        and a neighbourhood of $e$ in $X$, the exponential mapping. 
  \item [(b)] If $\text{char} (\mathbbm{k})=p$, there exists a flat structure on $X$ if $n$ is strictly less than $p$.  
  \item [(c)] Let $E\subset X$ be a smooth curve passing through $x$. 
        Then $E$ is flat for the flat structure of order $n$ if and only if the $m$-power operation 
        maps $E_{n,x}$ to $E_{n,x}$ for all integers $m$. 
 \end{enumerate}
\end{remark}
In Subsection \ref{infDef}, we will give an explicit description of the flat structure on the Picard variety $\Pic(C)$.

By abuse of notation, we call $\iota$ the equivariant isomorphism 
defining the flat structure at an arbitrary point $x\in X$. 
We use the isomorphism $\iota$ to expand regular functions on $X$ at $x$. More precisely,  
let $f_1,\dots, f_m$ be regular functions on $X$ at $x$ such that their pullback under $\iota$
forms a basis $\{x_1,\dots, x_m\}$ for the linear functions on $T_{n, 0}$. 

\begin{definition}
Let $g$ be a regular function on $X$ at $x$.
An \emph{expansion of $g$ in flat coordinates} is the expansion $\iota^*(g) = g_k + g_{k+1} + \cdots$  
where $g_j$ is a homogeneous polynomial in the variables $x_i$ on $T$. A \emph{component} of $f$ 
is such a homogeneous polynomial in the expansion. 
\end{definition}

Our main object of interest is the following.

\begin{definition} \label{oscCone}
 Let $X$ be an algebraic group with a flat structure and 
 let $Y\subset X$ be a smooth subvariety passing through $x\in X$. 
 The \emph{osculating cone of order $r$} to $Y$ at a point $x$, denoted by $\OC_r(Y, x) \subset T_{x}(X)$,
 is the closed scheme defined by the ideal generated by the forms 
 $$
 \{f_{k} |\ f_{k} \text{ is a component for an element } f\in I(Y), \forall k\leq r \}. 
 $$
\end{definition}

To get back to the Brill-Noether locus  $W^0_d(C)$, 
we end this section with an example.  

\begin{example}
 For the Brill-Noether locus $W^0_d(C)\subset \Pic^d(C)$, the tangent cone at $\Lscr\in W^1_d(C)$ coincides with 
 the osculating cone $\OC_2(W^0_d(C),\Lscr)$ of order $2$. 
 Indeed, by Corollary \ref{localEquationsW0d}, the Brill-Noether locus $W^0_d(C)$ is given locally 
 by the maximal minors of a $(d-1)\times 2$ matrix of regular functions vanishing at $L$. 
 Thus, the ideal of the tangent cone is generated by all quadratic components in the flat expansion of 
 regular functions in the ideal of $W^0_d(C)$.
\end{example}


\subsection{Infinitesimal deformations of global sections of $\Lscr$ } 
\label{infDef}

We follow \cite[Section $3$]{Kem83} in order to represent points of the canonical space $H^0(C,\omega_C)^*$ 
by principal parts of rational functions in $\OO C$. 
Then, we introduce flat coordinates on $\Pic(C)$ and show the connection of an infinitesimal deformation 
of $\Lscr$ and flat curves in $\Pic(C)$ as in \cite[Section 2]{Kem86}. 
Everything leads to a criterion for points in the canonical space to lie in the osculating cone. 
 \\
Let $\Mscr$ be an arbitrary line bundle on $C$ and let $\Rat(\Mscr)$ be the space of all 
rational sections of $\Mscr$. For a point $p\in C$, the space of principal parts of $\Mscr$ at $p$ 
is the quotient 
$$
\Prin_p(\Mscr)= \Rat(\Mscr) / \Rat_p(\Mscr),
$$
where $\Rat_p(\Mscr)$ is the space of rational sections of $\Mscr$ which are regular at $p$. 
Since a rational section of $\Mscr$ has only finitely many poles, we get a mapping 
\begin{align*}
\mathfrak{p}: \Rat(\Mscr) & \longrightarrow \Prin(\Mscr) := \bigoplus_{p\in C} \Prin_p(\Mscr) \\ 
               s & \longmapsto (s\ \text{modulo } \Rat_p(\Mscr))_{p\in C}
\end{align*}
and the following lemma holds.

\begin{lemma}{\cite[Lemma 3.3]{Kem83}}
\label{H0H1}
 The kernel and cokernel of $\mathfrak{p}$ are isomorphic to $H^0(C, \Mscr)$ and $H^1(C, \Mscr)$, respectively. 
\end{lemma}

In particular, an element $b\in H^1(C, \OO C)$ is represented by a collection $\beta = (\beta_p)_{p\in C}$ of 
rational functions, 
where $\beta_p$ is regular at $p$ except for finitely many $p$. 
 
We turn to infinitesimal deformations of our line bundle $\Lscr \in W^1_d(C)$ which are determined 
by elements in $H^1(C, \OO C)$. Furthermore, we will give an explicit description of the flat structure on $\Pic(C)$.
 
Let $X_i$ be the infinitesimal scheme $\Spec(A_i)$ supported on one point $x_0$, where $A_i$ is the Artinian ring 
$\mathbbm{k}[\varepsilon]/\varepsilon^{i+1}$ for $i\geq 1$.  
We consider the sheaf homomorphism $\OO C \longrightarrow \OO {C\times X_i} ^*$ given by 
the truncated exponential mapping 
$$
s \longmapsto 1 + s\varepsilon + \frac{s^2 \varepsilon^2}{2} + \dots + \frac{s^i\varepsilon^i}{i !}
              =:\exp_i(s\varepsilon), 
$$
which is the identity on $C\times \{x_0\}$. This homomorphism induces a map between cohomology groups 
$$
H^1(C, \OO C) \longrightarrow H^1(C\times X_i, \OO {C\times X_i}^*), 
$$ 
where the image of a cohomology class $b\in H^1(C, \OO C)$ determines a line bundle on $C\times X_i$, 
denoted by $\Lscr_i(b)$, and whose restriction to $C\times \{x_0\}$ is the structure sheaf $\OO C$. 
For $i=1$, this is the usual identification between $H^1(C, \OO C)$ and the tangent space to $\Pic(C)$ at $\OO C$.
Furthermore, there exists a unique morphism 
$$
\exp_i(b): X_i \rightarrow \Pic(C)
$$ 
such that $\exp_i(b)(x_0) = \OO C$ and 
$\Lscr_i(b)$ is the pullback of the Poincar\'e line bundle under the morphism $id_C \times \exp_i(b)$ 
by the universal property of the Poincar\'e line bundle. 
The image of $\exp_i(b)$ in $\Pic(C)$ is a flat curve. 
Indeed, for a general $i$, we have $\exp_i(b)^n = \exp_i(n b)$ for all integers $n$. 
By Remark \ref{flatCurve} (c), the image of $\exp_i(b)$ is a flat curve through $\OO C$ for $b\neq 0$ 
in the canonical flat structure on the algebraic group $\Pic(C)$.
 
After translation of the flat structure to the point $\Lscr$, we get flat curves 
$$
\exp_i(b): X_i \longrightarrow \Pic(C)
$$
with $\exp_i(b)(x_0) = \Lscr$ for $b \in H^1(C, \Lscr)$.  
We will describe the infinitesimal deformation of $\Lscr$ corresponding to such a flat curve in more detail. 
If $b\in H^1(C,\OO C)$ is represented by a collection $\beta = (\beta_p)_{p\in C}$ of rational functions, 
then the line bundle $\Lscr_i(b)$ is the $i$-th deformation of $\Lscr$ whose stalk at $p$ 
is given by rational sections $f= f_0 + f_1 \varepsilon + \cdots + f_i \varepsilon^i \in \Rat(\Lscr) \otimes A_i$ 
such that $f \exp(\beta_p \varepsilon)$ is regular at $p$. 
A different choice of $\beta$ gives a different but isomorphic subsheaf of $\Rat(\Lscr)\otimes A_i$ 
(see \cite[end of \S 2]{Kem86} for the details).

\begin{example}
\label{globalSectionOfL_2(b)}
If $i=2$, then $f=f_0 + f_1 \varepsilon + f_2 \varepsilon^2 \in \Rat(\Lscr) \otimes A_2$ 
is a global section of $\Lscr_2(b)$ if the following three conditions are satisfied:
 \begin{enumerate}
  \item [(a)] $f_0$ is a global section of $\Lscr$,
  \item [(b)] $f_1+f_0\beta_p$ is regular at $p$ for all $p\in C$ and 
  \item [(c)] $f_2+f_1\beta_p + f_0\beta_p^2/2$ is regular at $p$ for all $p\in C$. 
 \end{enumerate}
Note that the conditions are independent of the representative $\beta_p$. 
\end{example}

Furthermore, we have exact sequences 
\begin{align}\label{shortExactSequence}
  0 \longrightarrow \varepsilon \Lscr \longrightarrow & \Lscr_1(b) \longrightarrow \Lscr \longrightarrow 0, \\
 \nonumber
  0 \longrightarrow \varepsilon^2 \Lscr \longrightarrow & \Lscr_2(b) \longrightarrow \Lscr_1(b) \longrightarrow 0 
\emph{, etc.}.
\end{align}
Using the first exact sequence, 
we get a criterion for points in the canonical space lying in the tangent cone. 
We denote by $\mathbbm{k}\cdot b$ the line spanned by a cohomology class $b$. 

\begin{lemma} \label{criterionTangentCone}
 Let $0\neq b \in H^1(C, \OO C)$ be a nonzero cohomology class. 
 The point $\mathbbm{k}\cdot b\in\PP^{g-1}$ lies in the tangent cone if and only if 
 there exists a global section 
 $$
 f_0 + f_1 \varepsilon \in H^0(C\times X_1, \Lscr_1(b)) 
 $$
 where $f_0\neq 0$ is unique up to scalar.
\end{lemma}

\begin{proof}
 Let $0\neq b\in H^1(C, \OO C)$ be a cohomology class. 
 Applying the global section functor to the short exact sequence (\ref{shortExactSequence}), 
 we get the exact sequence 
 $$
 0 \longrightarrow H^0(C,\Lscr) \longrightarrow H^0(C\times X_1,\Lscr_1(b)) \longrightarrow H^0(C,\Lscr) 
 \stackrel{\cup b}{\longrightarrow} H^1(C,\Lscr) 
 $$
 where the coboundary map is given by the cup-product with $b$ by \cite[Lemma 10.6]{Kem83}. 
 By Remark \ref{tangentCone}, $\mathbbm{k}\cdot b$ is in the tangent cone if and only if the map $\cup b$ has rank $1$. 
 Thus, for points in the tangent cone,  $H^0(C\times X_1,\Lscr_1(b))$ is three-dimensional and there exists 
 a global section as desired. 
\end{proof}

The following criterion provides the connection of the osculating cone of order $3$ and second order deformations 
of $\Lscr$. The proof follows \cite[Lemma 4]{KS88}.

\begin{lemma} \label{criterionOscCone}
 Let $0\neq b \in H^1(C, \OO C)$ be a nonzero cohomology class. 
 The point $\mathbbm{k}\cdot b\in\PP^{g-1}$ lies in the osculating cone of order $3$ 
 if and only if there exists a global section
 $$
 f_0 + f_1 \varepsilon + f_2 \varepsilon^2 \in H^0(C\times X_2, \Lscr_2(b)) 
 $$
 where $f_0\neq 0$.
\end{lemma}

\begin{proof}
 Applying Theorem \ref{approxHom} to $\Lscr_2(b)$, there exists an approximating homomorphism of 
 $A_2$-modules 
 $$
 A_2^2 \stackrel{\varphi}{\longrightarrow} A_2^{d-1}
 $$
 such that $\ker (\varphi)=H^0(C\times X_2,\Lscr_2(b))$, $\coker (\varphi) = H^1(C\times X_2, \Lscr_2(b))$ and  
 the pullback of $W^0_d(C)$ via the flat curve $\exp_2(b): X_2 \longrightarrow \Pic(C)$ is given by 
 the maximal minors of $\varphi$. 
 In other words, $\varphi$ is the pullback of the matrix $(f_{ij})$ of Corollary \ref{localEquationsW0d}.
 The matrix $\varphi$ is equivalent to a matrix 
 $$
 \begin{pmatrix}
  \varepsilon^u & 0  & 0 & \cdots & 0\\
  0 & \varepsilon^v  & 0 & \cdots & 0\\ 
 \end{pmatrix}^T
 $$
 with $1 \leq u \leq v \leq 3$ since $\varphi \otimes k(\Lscr) = 0$. 
  \\
 Hence, the line $\mathbbm{k}\cdot b$ is contained in the osculating cone $\OC_3(W^0_d(C),\Lscr)$ if and only if $u+v\geq 4$. 
 Since the tangent cone is smooth and the point $b\neq 0$, the exponent $u=1$. 
 We conclude that the line $\mathbbm{k}\cdot b \in \OC_3(W^0_d(C),\Lscr)$ if and only if $v=3$. 
 Since $\varepsilon^3=0$, there exists a global section 
 $f=f_0 + f_1 \varepsilon + f_2 \varepsilon^2\in H^0(C\times X_2,\Lscr_2(b))$. 
  \\
 Restricting $\varphi$ to $\mathbbm{k} \cong A_0$, we get a commutative diagram
 $$
 \begin{xy}
  \xymatrix{
  0 \ar[r] & H^0(C\times X_2,\Lscr_2(b)) \ar[r] \ar[d] 
                    & A_2^2 \ar[r]^{\varphi} \ar[d] & A_2^{d-1} \\
  0 \ar[r] & H^0(C, \Lscr) \ar[r]^{\cong} & \mathbbm{k}^2 \ar[r]^{0} & \mathbbm{k}^{d-1} 
  }
 \end{xy}
 $$
 and thus, $f_0$ is nonzero in $H^0(C,\Lscr)$.
\end{proof}

Let $\mathbbm{k}\cdot b$ be a point in the tangent cone and 
let $f_0 + f_1 \varepsilon$ be the corresponding global section as in Lemma \ref{criterionTangentCone}. 
The following corollary answers the question of whether a second order deformation of the global section 
$f_0 \in H^0(C,\Lscr)$ is possible.

\begin{corollary}
\label{criterionOscCone2}
 Let $0\neq b = [\beta] \in H^1(C, \OO C)$ be a cohomology class, such that $\mathbbm{k}\cdot b$ lies in the tangent cone. 
 Then, $\mathbbm{k}\cdot b$ is contained in the osculating cone of order $3$ 
 if and only if the class $[\frac{f_0 \beta^2}{2}+ f_1\beta]$ is zero in $H^1(C,\Lscr) / (H^0(C,\Lscr)\cup [\beta])$ 
 for the section $f_0 + f_1 \varepsilon \in H^0(C\times X_1, \Lscr_1(b))$ of Lemma \ref{criterionTangentCone}.
\end{corollary}

\begin{proof}
The cohomology class $[\frac{f_0 \beta^2}{2} + f_1\beta]$ is zero in $H^1(C,\Lscr) / (H^0(C,\Lscr)\cup [\beta])$ 
if and only if $\frac{f_0 \beta^2}{2} + f_1\beta + f_2 + \widetilde{f} \beta$ 
is regular for an $\widetilde{f} \in H^0(C,\Lscr)$ and a rational section $f_2$. 
Taking Example \ref{globalSectionOfL_2(b)} into account, this condition
is satisfied if and only if 
$f_0 + (f_1+ \widetilde{f}) \varepsilon + f_2 \varepsilon^2 \in H^0(C\times X_2, \Lscr_2(b))$
since $f_0 + f_1 \varepsilon \in H^0(C\times X_1, \Lscr_1(b))$. 
Thus, it is equivalent for $b$ to be in the osculating cone of order $3$ by Lemma \ref{criterionOscCone}.
\end{proof}

In the proof of Theorem \ref{mainTheorem}, 
we will describe points in the fiber of the tangent cone over $\PP^1$ in terms of principal parts. 
We need the following lemma (see \cite[Section 3]{Kem86}). 

\begin{lemma}
\label{fiberTangentCone}
Let $f_0\in H^0(C,\Lscr)$ be a global section and let $D$ be its divisor of zeros. 
The fiber of the tangent cone $\overline{D}$ over the point $\mathbbm{k}\cdot f_0 \in \PP(H^0(C,\Lscr))$ is 
the projectivization of the kernel of  
$\cup f_0: H^1(C,\OO C) \rightarrow H^1(C,\OO C (D))$, denoted by $K(f_0)$. 
Furthermore, $K(f_0)$ is generated by cohomology classes of principal parts bounded by $D$, 
i.e., elements of $H^0(C,\OO C(D)|_{D})$ 
modulo the total principal part of elements in $H^0(C,\OO C(D))$.
\end{lemma}

\begin{proof}
 The first statement is clear by Remark \ref{tangentCone}. 
 For the second statement, we consider the short exact sequence 
 $0 \longrightarrow \OO C \longrightarrow \OO C(D) \longrightarrow \OO C(D)|_D \longrightarrow 0$
 and its long exact cohomology sequence 
 $$
 0 \rightarrow \mathbbm{k} \rightarrow H^0(C,\OO C(D)) \stackrel{\rho}{\longrightarrow} H^0(C,\OO C(D)|_D) 
   \rightarrow H^1(C, \OO C) \stackrel{\cup f_0}{\longrightarrow} H^1(C,\OO C(D)).
 $$
 Thus, $K(f_0)= \coker (H^0(C,\OO C(D)) \stackrel{\rho}{\longrightarrow}H^0(C,\OO C(D)|_D))$. 
 The map $\rho$ is the restriction of the map $\mathfrak{p}$ to the finite dimensional vector space $H^0(C,\OO C(D))$, i.e., 
 there is a commutative diagram 
 $$
 \begin{xy}
  \xymatrix{
  H^0(C,\OO C(D)) \ar[r]^{\rho} \ar@{^(->}[d] & H^0(C,\OO C(D)|_D) \ar@{^(->}[d] \\ 
  \Rat(\OO C(D)) \ar[r]^{\mathfrak{p}}& \Prin(\OO C(D)).
  }
 \end{xy}
 $$
 The second statement follows. 
\end{proof}

\section{Proof of the main theorem}
\label{proofOfTheMainTheorem}

Our proof is a modification of the proof in \cite[Section 4]{Kem86}. 
We fix the notation for the proof: 
Let   
\begin{center}
$
D=\sum\limits_{i=1}^{n} k_i p_i
$ 
\end{center}
be an arbitrary divisor in the linear system $|\Lscr|$, where $k_i\geq 1$ and $\sum_{i=1}^{n} k_i=d$. 
Let $f_0\in H^0(C,\Lscr)$ be the section whose divisor of zeros is exactly $D$ and let $(f_0,g_0)$ be 
a basis of $H^0(C,\Lscr)$. Let $h:=\frac{g_0}{f_0}\in H^0(C,\OO C(D)) \cong H^0(C,\Lscr)$. 
We now assume that each ramification point is tame, i.e., 
the ramification index $k_i\geq 2$ is coprime to the characteristic of $\mathbbm{k}$. 

\begin{proof}[Proof of Theorem \ref{mainTheorem}]
 We proceed as follows. We determine the condition on a point in the fiber to lie in the osculating cone and 
 reduce this condition to a system of equations. 
 To this end, we present the set of solutions as well as their geometry.
 
 \medskip

Let $b:=[\beta]\neq 0$ be an arbitrary point in the fiber $\overline{D}$ of the tangent cone, 
where $\beta=(\beta_p)_{p\in C} \in \Prin (\OO C(D))$. 
By Lemma \ref{fiberTangentCone}, we can choose the principal part $\beta$ such that 
$\beta_p$ is regular away from the support of $D$ and the pole order at $p_i$ is bounded by $k_i$, 
i.e., $\beta$ is spanned by elements of $H^0(C,\OO C(D)|_D)$.

\medskip

First of all, we state the condition that the point $b$ lies in the osculating cone. 
Since the line $\mathbbm{k}\cdot b$ spanned by $b$ is a point in a fiber $\overline{D}$ of the tangent cone, 
$f_0\beta$ is regular and $f_0 + 0\varepsilon$ is a global section of $H^0(C\times X_1, \Lscr_1(b))$ 
by Lemma \ref{criterionTangentCone}.
We can apply Corollary \ref{criterionOscCone2}. 
The point $\mathbbm{k}\cdot b\in \overline{D}$ is in the osculating cone $\OC_3(W^0_d(C),\Lscr)$ 
if and only if there exist sections $f_1,f_2\in \Rat(\Lscr)$ such that 
\begin{align}
\label{condition}
\frac{f_0\beta^2}{2} + f_1\beta + f_2 \text{\ \  is regular at $p_i$.}
\end{align}
Note that $f_1$ and $f_2$ are everywhere regular by Corollary \ref{criterionOscCone2} and Lemma \ref{propertyL2}, 
respectively. Thus, $f_1,f_2\in H^0(C,\Lscr)$ and 
the global section $f_1=a f_0 + c g_0$ is a linear combination of $f_0$ and $g_0$. 
Since $f_0\beta$ is regular, condition (\ref{condition}) simplifies: 
$$
\frac{f_0\beta^2}{2} + c g_0\beta \text{\ \ is regular at all $p_i$}
$$ 
for some constant $c\in \mathbbm{k}$.

Now, we reduce the condition to a simple system of equations. 
Since regularity is a local property, we study our condition at every single point $p_i$ of the support of $D$. 
To simplify notation we set $p:=p_i$ and $k:=k_i$.

Then, $\beta_{p}=\sum_{i=1}^{k} \lambda_i \beta_i$, 
where $\beta_i$ is the principal part of a rational function with pole of order $i$ at $p$. 
Hence, $\beta_1,\dots, \beta_k$ is a basis of $H^0(C,\OO C(D)|_{k\cdot p})$
and $\beta_{p}$ is an arbitrary linear combination of this basis. 

\medskip

We have to choose our basis of $H^0(C,\OO C(D)|_{k\cdot p})$ carefully 
in order to get the polar behaviour in condition (\ref{condition}) at $p$ under control. 
More precisely, we want to have equalities $\beta_j \beta_{k+i-j}=\beta_i \beta_k$ for $i\in\{1,\dots, k\}$ 
and $j\in\{i,\dots, k\}$.
 \\
An easy local computation shows the following claim which implies our desired equalities. 
Here, we need that $p$ is a tame point, i.e., the characteristic of $\mathbbm{k}$ does not divide $k$.

\begin{claim}
There exists a basis $\{\beta_i\}_{i=1,\dots, k}$ of $H^0(C,\OO C(D)|_{k\cdot p})$, 
i.e. the stalk of $\OO C(D)|_{k\cdot p}$ at $p$, 
satisfying the equations $\beta_k=h|_{p}$ and $\beta_i=(\beta_1)^i$ for all $i$.
\end{claim}

\begin{proof}\renewcommand{\qedsymbol}{}
A basis of the stalk of $\OO C(D)|_{k\cdot p}$ at $p$ is the images of $t^{-1}, t^{-2}, \dots, t^{-k}$ under 
$\mathfrak{p}:\Rat(\OO C(D))\to \Prin(\OO C(D))$ where $t$ is a local parameter function which vanishes simply at $p$.  
After rescaling our local parameter, we may assume that $\beta_k:=h|_{p}$ is of the form
$$
\beta_k=(c_1 t^{-1} + c_2 t^{-2} + \dots + c_{k-1} t^{-k+1} + t^{-k})|_{p}\in \OO C(D)|_{k\cdot p}
$$
for some constants $c_1,\dots,c_{k-1}\in \mathbbm{k}$. 
We define $\beta_1$ to be the expansion of the $k$-th root of the rational function  
$F:=c_1 t^{-1} + c_2 t^{-2} + \dots + c_{k-1} t^{-k+1} + t^{-k}$ up to some order. 
Therefore, we need the assumption that the characteristic of $\mathbbm{k}$ does not divide $k$. 
To be more precise, let  
\begin{align*}
\sqrt[k]{F} & = \sqrt[k]{c_1 t^{k-1} +  \dots + c_{k-1} t + 1}\cdot \frac{1}{t} \\
& = \frac{1}{t}\cdot \left(\sum_{j=0}^{k-1}\binom{1/k}{j}(c_1 t^{k-1} + \dots + c_{k-1} t)^j + h.o.t \right) 
\end{align*}
be the expansion of $\sqrt[k]{F}$.
We define 
$$
\beta_1:=
\frac{1}{t}\left(\sum_{j=0}^{k-1}\binom{1/k}{j}(c_1 t^{k-1} + c_2 t^{k-2} + \dots + c_{k-1} t)^j\right)\bigg \vert _{p}
$$
and $\beta_i:=(\beta_1)^i$ for $i=1,\dots, k-1$. 
Note that $\beta_i$ is the principal part of a rational function with pole of order $i$ and 
$\beta_1^k-\beta_k \in \Rat_p (\OO C(D))$. 
Hence, $\{\beta_i\}_{i=1,\dots, k}$ form a basis and $\beta_k=(\beta_1)^k$.
\end{proof}

Recall that $f_0 \beta_k=f_0\cdot\frac{g_0}{f_0}|_p=g_0|_p$ and 
$f_0 \beta_i\beta_j$ is regular for $i+j\leq k$.
Using our careful choice of $\beta_i$, condition (\ref{condition}) is fulfilled at the point $p$ if and only if  

\begin{align*}
\label{localCondition}
\frac{f_0(\beta_p)^2}{2} + c g_0\beta_p
= &\frac{f_0}{2}\left(\sum\limits_{i=1}^{k} \lambda_i \beta_i\right)^2 + 
c g_0 \left(\sum\limits_{i=1}^{k} \lambda_i \beta_i\right)\\
= &\frac{f_0}{2}\left(\sum\limits_{i=1}^{k}\sum\limits_{j=1}^{k} \lambda_i\lambda_j \beta_i \beta_j\right) + 
c g_0 \left(\sum\limits_{i=1}^{k} \lambda_i \beta_i\right) \\
= &\frac{f_0}{2}\left(\sum\limits_{i=1}^{k}\sum\limits_{j=i}^{k} \lambda_j\lambda_{k+i-j} \beta_j \beta_{k+i-j}\right) + 
c g_0 \left(\sum\limits_{i=1}^{k} \lambda_i \beta_i\right) \\ 
= &\frac{f_0}{2} \beta_k \left(\sum\limits_{i=1}^{k}
\left(\sum\limits_{j=i}^{k} \lambda_j\lambda_{k+i-j}\right)\beta_i \right) + 
c g_0 \left(\sum\limits_{i=1}^{k} \lambda_i \beta_i\right) \\ 
= &\frac{g_0}{2}\left( \sum_{i=1}^{k} 
\left(\sum_{j=i}^{k} \lambda_j \lambda_{k+i-j} + 2c\lambda_i \right)\beta_i\right) \in \OO C(D)|_{k\cdot p} \\
  & \text{ is regular at $p$.} 
\end{align*}
Since the global section $g_0$ does not vanish at $p\in \Supp(D)$, 
condition (\ref{condition}) is regular at $p$ if and only if 
\vspace{-3mm}
\begin{center}
$
\sum\limits_{j=i}^{k} \lambda_j \lambda_{k+i-j} + 2c\lambda_i=0
$
\end{center}
\vspace{-3mm}
for all $i=1,\dots,k$. 

At the end, we have to solve this system of equations and describe the geometry. 
Let $\lambda_{i}, c$ be a solution of the equations.
In order to relate a solution to its geometry, we distinguish two cases. 
The second case is only relevant if the multiplicity $k\geq 2$.

\begin{case}
\label{case1}
If $c \neq 0$, then either $\beta_p = -2c \beta_k=-2c\cdot h|_{p}$ or $\beta_p=0$. 
Geometrically, either $\mathbbm{k}\cdot [\beta_p] = \mathbbm{k}\cdot [h|_p] \in \overline{k\cdot p}$ 
or $[\beta_p]=0$. 
\end{case}

\begin{proof}\renewcommand{\qedsymbol}{} Let $c$ be a nonzero constant. 
We consider the equation 
\begin{center}
$
\lambda_k^2+2c\lambda_k=0.
$
\end{center}
Hence, either $\lambda_k=-2c$ or $\lambda_k=0$. For both solutions, the equation 
\begin{center}
$
2\lambda_{k-1}\lambda_k + 2c\lambda_{k-1}=0
$
\end{center}
implies that $\lambda_{k-1}=0$. Inductively, $\lambda_i=0$ for $1\leq i \leq k-1$, which proves our claim.  
\end{proof}
\begin{case}
\label{case2}
If $c = 0$, then $\beta_p = \sum\limits_{i=1}^{\lfloor \frac{k}{2} \rfloor} \lambda_i \beta_i$ 
is an arbitrary linear combination. 
Geometrically, $\mathbbm{k}\cdot [\beta_p] \in \overline{\lfloor \frac{k}{2} \rfloor \cdot p}$.
\end{case}

\begin{proof}\renewcommand{\qedsymbol}{} 
If $c=0$, then the system of equation reduces to 
\begin{align*}
 &\begin{cases} \lambda_k^2=0 & \hspace{7.7cm} (i\!=\!k) \\
                \lambda_{k-1}\lambda_k + \lambda_k\lambda_{k-1} = 0 & \hspace{7.7cm}  (i\!=\!k\!-\!1) 
  \end{cases} \\
 &\begin{cases}\lambda_{k-2}\lambda_k + \lambda_{k-1}^2 + \lambda_k\lambda_{k-2}=0 & \hspace{3.5cm} (i\!=\!k\!-\!2) \\
               \lambda_{k-3}\lambda_k + \lambda_{k-2} \lambda_{k-1} + \lambda_{k-1}\lambda_{k-2} + 
               \lambda_k\lambda_{k-3} = 0 & \hspace{3.5cm} (i\!=\!k\!-\!3)
  \end{cases} \\ 
 & \ \ \ \ \ \ \vdots \\ 
 &\begin{cases}\lambda_{k-2l}\lambda_k + \dots + \lambda_{k-l}^2 + \dots + \lambda_k\lambda_{k-2l}=0 & (i\!=\!k\!-\!2l) \\
               \lambda_{k-2l-1}\lambda_k + \dots + \lambda_{k-l-1} \lambda_{k-l} + 
               \lambda_{k-l}\lambda_{k-l-1} + \dots + \lambda_k\lambda_{k-2l-1} = 0 & (i\!=\!k\!-\!2l\!-\!1)
  \end{cases} \\ 
 & \ \ \ \ \ \ \vdots \\
 & \ \ \ \lambda_{1}\lambda_k + \dots + \lambda_{\lfloor \frac{k+1}{2} \rfloor} \lambda_{\lceil \frac{k+1}{2} \rceil} + 
        + \dots + \lambda_k\lambda_{1} = 0 \hspace{3.9cm}  (i\!=\!1). 
\end{align*}

Thus, $\lambda_k=0$ solve the first two equations. The second pair now gives $\lambda_k=\lambda_{k-1}=0$.
Inductively, we obtain $\lambda_{k-l}=0$ for all $l$ with $k-2l\geq 1$. 
Whether the last equation gives a condition depends on the parity of $k$. 
Thus 
$\lambda_k=\dots=\lambda_{\lfloor \frac{k}{2} \rfloor +1} = 0$ and 
$\lambda_1,\dots, \lambda_{\lfloor \frac{k}{2} \rfloor}$ arbitrary is the solution of the system of equations. 
\end{proof}

This completes the local study of condition (\ref{condition}).
We now make use of the local description of $\beta$ which leads to a global picture. 
The principal part $\beta=\sum_{i=1}^n \beta_{p_i}$ is supported on $D$. 
Since the constant $c$ is the same for all local computations,   
either all principal parts $\beta_{p_i}$ satisfy Case \ref{case1} or all principal parts satisfy Case \ref{case2}. 

By Lemma \ref{fiberTangentCone}, the total principal part of $h$ yields a relation
\begin{align}
 \bigg[\sum_{i=1}^{n} h|_{p_i}\bigg]=0 \tag{$\ast$}.  
\end{align}

If we are in Case \ref{case1}, let 
$\emptyset\neq I\subsetneq\{1,\dots, n\}$ be the set of indices, where $[\beta_{p_i}]=[h|_{p_i}]\neq 0$ 
and let $I^c=\{1,\dots, n\}\backslash I$ be its complement, 
then the nonzero point $b=[\beta]$ lies in the osculating cone if and only if 
$$
\mathbbm{k}\cdot b = \mathbbm{k}\cdot\bigg[\sum\limits_{i\in I} h|_{p_i}\bigg] \stackrel{(\ast)}{=} 
\mathbbm{k}\cdot \bigg[\sum\limits_{i\in I^c} h|_{p_i}\bigg] 
\in \underbrace{\overline{\sum\limits_{i\in I} k_i p_i}}_{=:D_1}\bigcap 
    \underbrace{\overline{\sum\limits_{i\in I^c} k_i p_i}}_{=:D_2},
$$ 
In Case \ref{case2}, the nonzero point $b=[\beta]$ lies in the osculating cone if and only if
$$
\mathbbm{k} \cdot b \in \overline{\sum\limits_{i=1}^n \bigg\lfloor \frac{k_i}{2} \bigg\rfloor p_i} = 
\underbrace{\overline{\sum\limits_{i=1}^n \bigg\lfloor \frac{k_i}{2} \bigg\rfloor p_i}}_{=:D_1}\bigcap 
\underbrace{\overline{\sum\limits_{i=1}^n \bigg\lceil \frac{k_i}{2} \bigg\rceil p_i}}_{=:D_2}.
$$ 
Note that $\overline{\sum_{i=1}^n \lfloor \frac{k_i}{2} \rfloor p_i}$ is the greatest 
linear span of osculating spaces to $C$ at $p_i$ which 
can be expressed as an intersection $\overline{D_1}\cap \overline{D_2}$ of a nonzero effective decomposition $D=D_1+D_2$. 
Our theorem follows. 
\end{proof}

\begin{remark}
The two cases, appearing in the proof, contribute differently to the osculating cone. 
From the first case, we always get $2^{n-1}-1$ points in the osculating cone. 
If the curve $C$ has at least one ramification point in the fiber, 
there is a $(\sum_{i=1}^n \lfloor \frac{k_i}{2} \rfloor-1)$-dimensional component in the fiber by the second case. 
The osculating cone has a higher-dimensional component in a fiber $\overline{D}$ of the tangent cone 
unless $D=p_1+\dots+p_d$ or $D=2p_1+p_2+\dots+p_{d-1}$ or $D=3p_1+p_2+\dots+p_{d-2}$. 
\end{remark}

\begin{example}[Example \ref{exampleOC} continued]
 In a fiber over a general point $(\lambda: \mu)\in \PP^1$, the curve $C$ is unramified. 
 In a general fiber, which is a projective plane, 
 the four distinct points of $C$ determine three pairs of connection lines and 
 hence, three corresponding intersection points. 
 Thus, all intersection points, associated to the pencil $|\Lscr|$ on $C$, sweep out a trigonal curve $C_2$ 
 (see also \cite[Section 12.7]{BL}). 
 We assume that $\pi_C: C\to \PP^1$ has only simple branch points. 
 By the Riemann-Hurwitz formula 
 $$
 2\cdot 6-2 - 4\cdot(2\cdot 0-2) = 18, 
 $$
 the curve $C$ has $18$ simple ramification points.  
 Applying Theorem \ref{mainTheorem} for $D=2p_1+p_2+p_3$, 
 there are $4$ points in a special fiber as mentioned above. 

 
By the geometry of these points, every simple ramification point of $C$ is a simple ramification point of $C_2$, too. 
Hence, the genus $g(C_2)$ of $C_2$ is 
$$
g(C_2)= \frac{3\cdot(-2) + 18}{2} +1 = 7 
$$
by the Riemann-Hurwitz formula.
Furthermore, $C$ and $C_2$ intersect transversally since 
\begin{align*}
 \chi(C\cap C_2) =  & \chi(C) + \chi(C_2) - \chi(C\cup C_2)  \\ 
                =  & \chi(C) + \chi(C_2) - \chi(\OC_3(W^0_4(C),\Lscr)) = 1-6+ 1-7 - (1-30)= 18.
\end{align*} 
The osculating cone is thus the union of two transversal intersecting curves of genus $6$ and genus $7$. 

Note that the space of the six connection lines induce 
an \'etale double cover $\widetilde{C_2} \to C_2$ of the trigonal curve $C_2$. 
By Recillas' theorem (\cite{R74}), the Jacobian of $C$ and the Prym variety associated to $\widetilde{C_2} \to C_2$ 
are isomorphic.
\end{example}

\begin{remark}
We now consider a general curve of arbitrary genus $g\geq 5$ and 
a point $\Lscr$ of a Brill-Noether locus $W^1_d(C)$ of dimension at least $1$. 
Then, Lemma \ref{criterionTangentCone} and Lemma \ref{criterionOscCone} are still true.   
With the methods used in the proof of the main theorem, one can show that for a general divisor $D\in |\Lscr|$, 
the intersection of the osculating cone and the linear span $\overline{D}$ contains 
all intersection points of the form $\overline{D_1}\cap \overline{D_2}$ 
for any decomposition $D=D_1+D_2$ into nonzero effective divisors. 
\end{remark}


 

\bibliography{papers}{}

\begin{thebibliography}{ACGH85}

\bibitem[ACGH85]{ACGH}
E.~Arbarello, M.~Cornalba, P.~A. Griffiths, and J.~Harris.
\newblock {\em Geometry of algebraic curves. {V}ol. {I}}, volume 267 of {\em
  Grundlehren der Mathematischen Wissenschaften [Fundamental Principles of
  Mathematical Sciences]}.
\newblock Springer-Verlag, New York, 1985.

\bibitem[BL04]{BL}
Christina Birkenhake and Herbert Lange.
\newblock {\em Complex abelian varieties}, volume 302 of {\em Grundlehren der
  Mathematischen Wissenschaften [Fundamental Principles of Mathematical
  Sciences]}.
\newblock Springer-Verlag, Berlin, second edition, 2004.

\bibitem[CS95]{CS95}
Ciro Ciliberto and Edoardo Sernesi.
\newblock Singularities of the theta divisor and families of secant spaces to a
  canonical curve.
\newblock {\em J. Algebra}, 171(3):867--893, 1995.

\bibitem[CS00]{CS00}
Ciro Ciliberto and Edoardo Sernesi.
\newblock On the geometry of canonical curves of odd genus.
\newblock {\em Comm. Algebra}, 28(12):5993--6001, 2000.
\newblock Special issue in honor of Robin Hartshorne.

\bibitem[Dal85]{Dal85}
S.~G. Dalalyan.
\newblock On tetragonal curves.
\newblock In {\em Mathematics, {N}o.\ 3 ({R}ussian)}, pages 64--81. Erevan.
  Univ., Erevan, 1985.

\bibitem[GH80]{GH80}
Phillip Griffiths and Joseph Harris.
\newblock On the variety of special linear systems on a general algebraic
  curve.
\newblock {\em Duke Math. J.}, 47(1):233--272, 1980.

\bibitem[Gie82]{Gie82}
David Gieseker.
\newblock Stable curves and special divisors: {P}etri's conjecture.
\newblock {\em Invent. Math.}, 66(2):251--275, 1982.

\bibitem[Gro63]{EGA}
Alexander Grothendieck.
\newblock \'{E}l\'ements de g\'eom\'etrie alg\'ebrique. {III}. \'{E}tude
  cohomologique des faisceaux coh\'erents. {II}.
\newblock {\em Inst. Hautes \'Etudes Sci. Publ. Math.}, (17):91, 1963.

\bibitem[Hel01]{H01}
Sigurdur Helgason.
\newblock {\em Differential geometry, {L}ie groups, and symmetric spaces},
  volume~34 of {\em Graduate Studies in Mathematics}.
\newblock American Mathematical Society, Providence, RI, 2001.
\newblock Corrected reprint of the 1978 original.

\bibitem[Kem73]{Kem73}
George Kempf.
\newblock On the geometry of a theorem of {R}iemann.
\newblock {\em Ann. of Math. (2)}, 98:178--185, 1973.

\bibitem[Kem83]{Kem83}
George~R. Kempf.
\newblock {\em Abelian integrals}, volume~13 of {\em Monograf\'\i as del
  Instituto de Matem\'aticas [Monographs of the Institute of Mathematics]}.
\newblock Universidad Nacional Aut\'onoma de M\'exico, M\'exico, 1983.

\bibitem[Kem86]{Kem86}
George~R. Kempf.
\newblock The equations defining a curve of genus {$4$}.
\newblock {\em Proc. Amer. Math. Soc.}, 97(2):219--225, 1986.

\bibitem[KS88]{KS88}
George~R. Kempf and Frank-Olaf Schreyer.
\newblock A {T}orelli theorem for osculating cones to the theta divisor.
\newblock {\em Compositio Math.}, 67(3):343--353, 1988.

\bibitem[May13]{Mayer}
Ulrike Mayer.
\newblock Osculating cones to {B}rill-{N}oether loci.
\newblock {\em Master thesis, Saarland University}, 2013.

\bibitem[Rec74]{R74}
Sevin Recillas.
\newblock Jacobians of curves with {$g^{1}_{4}$}'s are the {P}rym's of trigonal
  curves.
\newblock {\em Bol. Soc. Mat. Mexicana (2)}, 19(1):9--13, 1974.

\end{thebibliography}
\bibliographystyle{alpha}

\end{document}